\documentclass[reqno]{amsart}
\usepackage{amsmath}
\usepackage[dvips]{graphicx}
\usepackage{amsfonts}
\usepackage{amssymb}
\usepackage{latexsym}
\graphicspath{{Img/}}

\newtheorem{theorem}{Theorem}

\theoremstyle{plain}

\newtheorem{corollary}{Corollary}

\newtheorem{example}{Example}

\newtheorem{lemma}{Lemma}

\newtheorem{problem}{Problem}
\newtheorem{proposition}{Proposition}
\newtheorem{remark}{Remark}

\numberwithin{equation}{section}

\begin{document}
\title[Analytical solution of the weighted Fermat-Torricelli problem]{Analytical solution of the weighted Fermat-Torricelli problem for tetrahedra: The case of two pairs of equal weights}
\author{Anastasios N. Zachos}
\address{University of Patras, Department of Mathematics, GR-26500 Rion, Greece}
\email{azachos@gmail.com}

\begin{abstract}
The weighted Fermat-Torricelli problem for four non-collinear and
non-coplanar points in $\mathbb{R}^{3}$ states that:

Given four non-collinear and non-coplanar points $A_{1},$ $A_{2},$
$A_{3},$ $A_{4}$ and a positive real number (weight) $B_{i}$ which
correspond to each point $A_{i},$ for $i=1,2,3,4,$ find a fifth
point such that the sum of the weighted distances to these four
points is minimized. We present an analytical solution for the
weighted Fermat-Torricelli problem for tetrahedra in
$\mathbb{R}^{3}$ for the case of two pairs of equal weights.

\end{abstract}\maketitle

\section{Introduction}

Let  $A_{1},$ $A_{2},$ $A_{3},$ $A_{4}$ be four non-collinear and
non-coplanar points and a positive real number (weight) $B_{i}$
correspond to each point $A_{i},$ for $i=1,2,3,4.$

The weighted Fermat-Torricelli problem for four non-collinear
points and non-coplanar points in $\mathbb{R}^{3}$ states
that:

\begin{problem} Find a unique (fifth) point $A_{0}\in \mathbb{R}^{3},$
which minimizes

\[f(X)=\sum_{i=1}^{4}B_{i}\|X-A_{i}\|,\]
where $\|\cdot\|$ denotes the Euclidean distance and $X\in
\mathbb{R}^{3}.$

\end{problem}

The existence and uniqueness of the weighted Fermat-Torricelli
point and a complete characterization of the solution of the
weighted Fermat-Torricelli problem for tetrahedra has been
established in \cite[Theorem~1.1, Reformulation~1.2 page~58,
Theorem~8.5 page 76, 77]{Kup/Mar:97}).

\begin{theorem}{\cite{BolMa/So:99},\cite{Kup/Mar:97},\cite{KupitzMartini:94}}\label{theor1}
Let there be given four non-collinear points and non-coplanar
points $\{A_{1},A_{2},A_{3},A_{4}\},$ $A_{1}, A_{2},
A_{3},A_{4}\in\mathbb{R}^{3}$ with corresponding positive
weights $B_{1}, B_{2}, B_{3}, B_{4}.$ \\
(a) The weighted Fermat-Torricelli point $A_{0}$ exists and is
unique. \\
(b) If for each point $A_{i}\in\{A_{1},A_{2},A_{3},A_{4}\}$

\begin{equation}\label{floatingcase}
\|{\sum_{j=1, i\ne j}^{4}B_{j}\vec u(A_i,A_j)}\|>B_i,
\end{equation}

 for $i,j=1,2,3$  holds,
 then \\
 ($b_{1}$) the weighted Fermat-Torricelli point $A_{0}$ (weighted floating equilibrium point) does not belong to $\{A_{1},A_{2},A_{3},A_{4}\}$
 and \\
 ($b_{2}$)

\begin{equation}\label{floatingequlcond}
 \sum_{i=1}^{4}B_{i}\vec u(A_0,A_i)=\vec 0,
\end{equation}
where $\vec u(A_{k} ,A_{l})$ is the unit vector from $A_{k}$ to
$A_{l},$ for $k,l\in\{0,1,2,3,4\}$
 (Weighted Floating Case).\\
 (c) If there is a point $A_{i}\in\{A_{1},A_{2},A_{3},A_{4}\}$
 satisfying
 \begin{equation}
 \|{\sum_{j=1,i\ne j}^{4}B_{j}\vec u(A_i,A_j)}\|\le B_i,
\end{equation}
then the weighted Fermat-Torricelli point $A_{0}$ (weighted
absorbed point) coincides with the point $A_{i}$ (Weighted
Absorbed Case).
\end{theorem}

We consider the following open problem:

\begin{problem}
Find an analytic solution with respect to the weighted
Fermat-Torricelli problem for tetrahedra in $\mathbb{R}^{3},$ such
that the corresponding weighted Fermat-Torricelli point is not any
of the given points.
\end{problem}

In this paper, we present an analytical solution for the weighted
Fermat-Torricelli problem for regular tetrahedra  in
$\mathbb{R}^{3}$ for $B_{1}>B_{4},$ $B_{1}=B_{2}$ and
$B_{3}=B_{4},$ by expressing the objective function as a function
of the linear segment which is formed by the middle point of the
common perpendicular $A_{1}A_{2}$ and $A_{3}A_{4},$  and the
corresponding weighted Fermat-Torricelli point $A_{0}$ (Section~2,
Theorem~\ref{theortetr}). It is worth mentioning that this
analytical solution of the weighted Fermat-Torricelli problem for
a regular tetrahedron is a generalization of the analytical
solution of the weighted Fermat-Torricelli point of a quadrangle
(tetragon) in $\mathbb{R}^{2}$ (see in \cite{Zachos:14}).

By expressing the angles $\angle A_{i}A_{0}A_{j}$ for
$i,j=1,2,3,4$ for $i\ne j$  as a function of $B_{1},$ $B_{4}$ and
$a$ and taking into account the invariance property of the
weighted Fermat-Torricelli point (geometric plasticity) in
$\mathbb{R}^{3},$ we obtain an analytical solution for some
tetrahedra having the same weights with the regular tetrahedron
(Section~3, Theorem~\ref{theorquadnntetrah}).

\section{The weighted Fermat-Torricelli problem for regular tetrahedra: The case  $B_{1}=B_{2}$ and $B_{3}=B_{4}.$ }

We shall consider the weighted Fermat-Torricelli problem for a
regular tetrahedron $A_{1}A_{2}A_{3}A_{4},$ for $B_{1}>B_{4},$
$B_{1}=B_{2}$ and $B_{3}=B_{4}.$

We denote by $a_{ij}$ the length of the linear segment $A_iA_j,$
by $A_{12}A_{34}$ the common perpendicular of $A_{1}A_{2}$ and
$A_{3}A_{4}$ where $A_{12}$ is the middle point of $A_{1}A_{2}$
and $A_{34}$ is the middle point of $A_{3}A_{4},$ by $A_{0}$ the
weighted Fermat-Torricelli point of $A_{1}A_{2}A_{3}A_{4}$ by $O$
the middle point of $A_{12}A_{34}$ ($A_{12}O=A_{34}O$), by $y$ the
length of the linear segment $OA_{0}$ and $\alpha_{ikj}$ the angle
$\angle A_{i}A_{k}A_{j}$ for $i,j,k=0,1,2,3,4, i\neq j\neq k.$ We
set $a_{ij}\equiv a,$ the edges of $A_{1}A_{2}A_{3}A_{4}$
(Fig.~\ref{fig1}).

\begin{figure}
\centering
\includegraphics[scale=0.2]{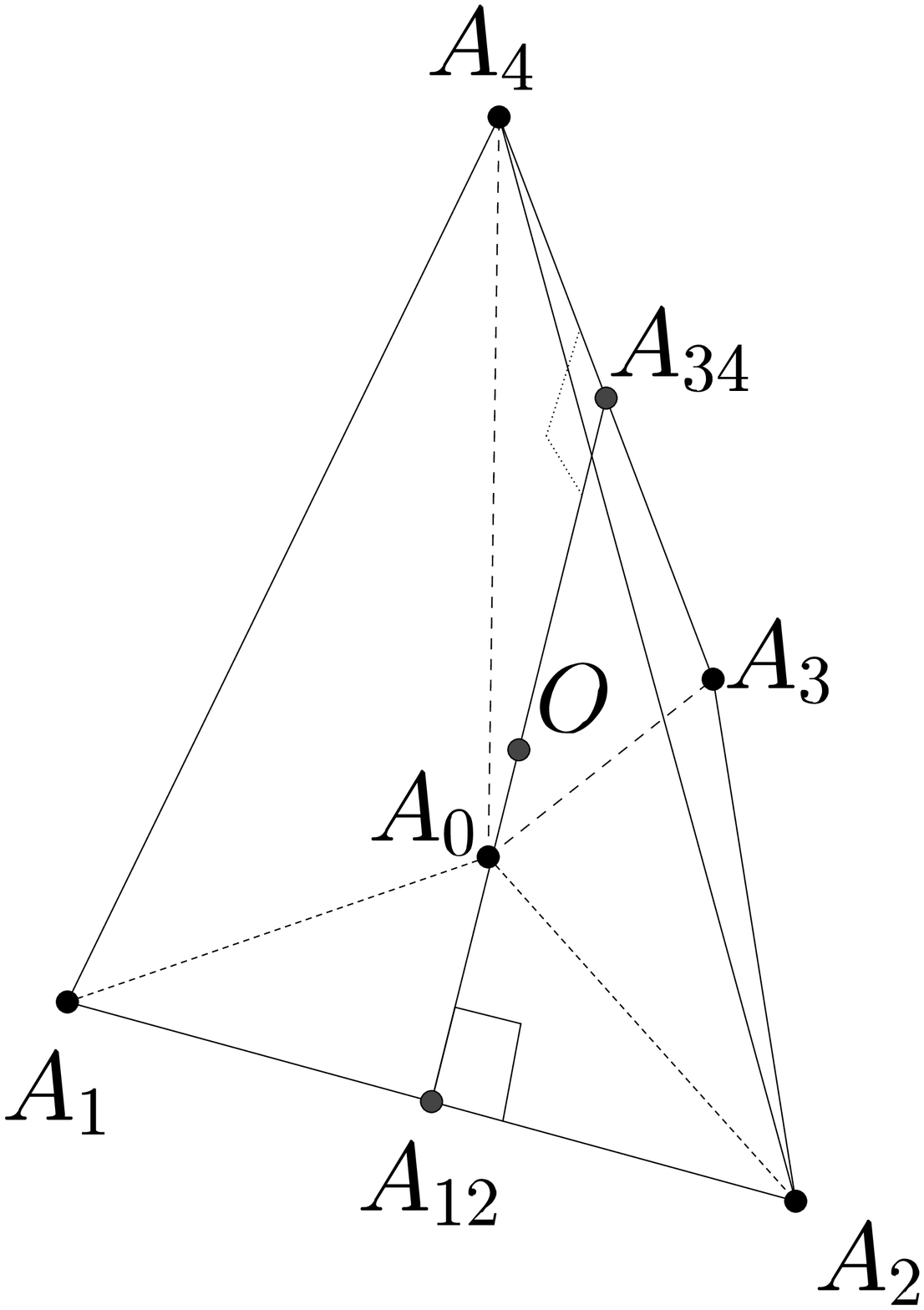}
\caption{The weighted Fermat-Torricelli problem for a regular
tetrahedron $B_{1}=B_{2}$ and $B_{3}=B_{4}$ for
$B_{1}>B_{4}$}\label{fig1}
\end{figure}

\begin{problem}\label{sym1}
Given a regular tetrahedron $A_{1}A_{2}A_{3}A_{4}$ and a weight
$B_{i}$ which corresponds to the vertex $A_{i},$ for $i=1,2,3,4,$
find a fifth point $A_{0}$ (weighted Fermat-Torricelli point)
which minimizes the objective function

\begin{equation}\label{obj1}
f=B_{1}a_{01}+B_{2} a_{02}+ B_{3} a_{03}+B_{4} a_{04}
\end{equation}
for $B_{1}>B_{4},$ $B_{1}=B_{2}$ and $B_{3}=B_{4}.$
\end{problem}

We set

\begin{eqnarray}\label{analsolrtetrahedron1}
&&s\equiv -a^6 B_1^{12}+2 a^6 B_1^{10} B_4^2+a^6 B_1^8 B_4^4-4 a^6
B_1^6 B_4^6+a^6 B_1^4 B_4^8+2 a^6 B_1^2 B_4^{10}-a^6
B_4^{12}+2 \sqrt{2}\nonumber{}\\
&&{}\surd(a^{12} B_1^{22} B_4^2-8 a^{12} B_1^{20} B_4^4+29 a^{12}
B_1^{18} B_4^6-64 a^{12} B_1^{16} B_4^8+98 a^{12} B_1^{14}
B_4^{10}-112 a^{12} B_1^{12} B_4^{12}+\nonumber{}\\
&&{}+98 a^{12} B_1^{10} B_4^{14}-64 a^{12} B_1^8 B_4^{16}+29
a^{12} B_1^6 B_4^{18}-8 a^{12} B_1^4 B_4^{20}+a^{12} B_1^2
B_4^{22})
\end{eqnarray}

and

\begin{eqnarray}\label{analsolrtetrahedron2}
t\equiv -\frac{a^4 B_1^4}{4 s^{1/3}}+\frac{a^4 B_1^2 B_4^2}{2
s^{1/3}}-\frac{a^4 B_4^4}{4 s^{1/3}}-\frac{s^{1/3}}{4
\left(B_1^4-2 B_1^2 B_4^2+B_4^4\right)}.
\end{eqnarray}


\begin{theorem}\label{theortetr}
The location of the weighted Fermat-Torricelli point $A_{0}$ of
$A_{1}A_{2}A_{3}A_{4}$ for $B_{1}=B_{2},$ $B_{3}=B_{4}$ and
$B_{1}>B_{4}$ is given by:

\begin{eqnarray}\label{analsolrtetrahedron}
&&y=-\frac{\sqrt{t}}{2}+\nonumber{}\\
&&{}\frac{1}{2} \sqrt{\frac{a^4 B_1^4}{4 s^{1/3}}-\frac{a^4 B_1^2
B_4^2}{2 s^{1/3}}+\frac{a^4 B_4^4}{4 s^{1/3}}+\frac{2 \left(-8
\sqrt{2} a^3 B_1^2-8 \sqrt{2} a^3 B_4^2\right)}{\sqrt{t} \left(64
B_1^2-64 B_4^2\right)}+\frac{s^{1/3}}{4 \left(B_1^4-2 B_1^2
B_4^2+B_4^4\right)}}\nonumber{}\\
\end{eqnarray}

\end{theorem}

\begin{proof}[Proof of Theorem~\ref{theortetr}:]

Taking into account the symmetry of the weights $B_{1}=B_{4}$ and
$B_{2}=B_{3}$ for $B_{1}>B_{4}$ and the symmetries of the regular
tetrahedron $A_{1}A_{2}A_{3}A_{4}$ the objective function
(\ref{obj1}) of the weighted Fermat-Torricelli problem
(Problem~\ref{sym1}) could be reduced to an equivalent Problem:
Find a point $A_{0}$ which belongs to the midperpendicular
$A_{12}A_{34}$ of $A_{1}A_{2}$ and $A_{3}A_{4}$ and minimizes the
objective function

\begin{equation}\label{obj12}
\frac{f}{2}=B_{1}a_{01}+B_{4} a_{04}.
\end{equation}

We express $a_{01},$ $a_{02},$ $a_{03}$ and $a_{04}$ as a function
of $y:$

\begin{equation}\label{a01}
a_{01}^{2}=\left(\frac{a}{2}\right)^{2}+\left(\frac{a\frac{\sqrt{2}}{2}}{2}-y\right)^{2},
\end{equation}

\begin{equation}\label{a02}
a_{02}^{2}=\left(\frac{a}{2}\right)^{2}+\left(\frac{a\frac{\sqrt{2}}{2}}{2}-y\right)^{2},
\end{equation}

\begin{equation}\label{a03}
a_{03}^{2}=\left(\frac{a}{2}\right)^{2}+\left(\frac{a\frac{\sqrt{2}}{2}}{2}+y\right)^{2},
\end{equation}

\begin{equation}\label{a04}
a_{04}^{2}=\left(\frac{a}{2}\right)^{2}+\left(\frac{a\frac{\sqrt{2}}{2}}{2}+y\right)^{2},
\end{equation}

where the length of $A_{12}A_{34}$ is $\frac{a\sqrt{2}}{2}.$

By replacing (\ref{a01}) and (\ref{a04}) in (\ref{obj12}) we get:

\begin{equation}\label{obj13}
B_{1}\sqrt{\left(\frac{a}{2}\right)^{2}+\left(\frac{a\frac{\sqrt{2}}{2}}{2}-y\right)^{2}}+B_{4}\sqrt{\left(\frac{a}{2}\right)^{2}+\left(\frac{a\frac{\sqrt{2}}{2}}{2}+y\right)^{2}}
\to min.
\end{equation}

By differentiating (\ref{obj13}) with respect to $y,$ and by
squaring both parts of the derived equation, we get:

\begin{equation}\label{fourth1}
\frac{B_{1}^2 \left(\frac{a\frac{\sqrt{2}}{2}}{2}-y\right)^{2}
}{\left(\frac{a}{2}\right)^{2}+\left(\frac{a\frac{\sqrt{2}}{2}}{2}-y\right)^{2}}=\frac{B_{4}^2
\left(\frac{a\frac{\sqrt{2}}{2}}{2}+y\right)^{2}}{\left(\frac{a}{2}\right)^{2}+\left(\frac{a\frac{\sqrt{2}}{2}}{2}+y\right)^{2}}
\end{equation}

which yields

\begin{equation}\label{fourth2}
64 y^4 \left(B_1^2-B_4^2\right)-8 \sqrt{2} a^3 y
\left(B_1^2+B_4^2\right)+3 a^4\left(B_1^2-B_4^2\right)=0.
\end{equation}

By solving the fourth order equation (\ref{fourth2}) with respect
to $y$ , we derive two complex solutions and two real solutions
(see in \cite{Shmakov:11} for the general solution of a fourth
order equation with respect to $y$) which depend on $B_{1}, B_{4}$
and $a.$ One of the two real solutions with respect to $y$ is
(\ref{analsolrtetrahedron}). The real solution
(\ref{analsolrtetrahedron}) gives the location of the weighted
Fermat-Torricelli point $A_{0}$ at the interior of
$A_{1}A_{2}A_{3}A_{4}$ (see fig.~\ref{fig1}).

\end{proof}


We shall state the Complementary weighted Fermat-Torricelli
problem for a regular tetrahedron (\cite[pp.~358]{Cour/Rob:51}),
in order to explain the second real solution which have been
obtained by (\ref{fourth2}) with respect to $y$ (see also in
\cite{Zachos:14} for the case of a quadrangle).

\begin{problem}\label{sym2complementary}
Given a regular tetrahedron $A_{1}A_{2}A_{3}A_{4}$ and a weight
$B_{i}$ (a positive or negative real number) which corresponds to
the vertex $A_{i},$ for $i=1,2,3,4,$ find a fifth point $A_{0}$
(weighted Fermat-Torricelli point) which minimizes the objective
function

\begin{equation}\label{obj1}
f=B_{1}a_{01}+B_{2} a_{02}+ B_{3} a_{03}+B_{4} a_{04}
\end{equation}
for $\|B_{1}\|>\|B_{4}\|,$ $B_{1}=B_{2}$ and $B_{3}=B_{4}.$
\end{problem}

\begin{proposition}\label{theortetrcomp1}
The location of the complementary weighted Fermat-Torricelli point
$A_{0}^{\prime}$ (solution of Problem~\ref{sym2complementary}) of
the regular tetrahedron $A_{1}A_{2}A_{3}A_{4}$ for
$B_{1}=B_{2}<0,$ $B_{3}=B_{4}<0$ and $\|B_{1}\|>\|B_{4}\|$ is the
exactly same with the location of the corresponding weighted
Fermat-Torricelli point $A_{0}$ of $A_{1}A_{2}A_{3}A_{4}$ for
$B_{1}=B_{2}>0,$ $B_{3}=B_{4}>0$ and $\|B_{1}\|>\|B_{4}\|.$

\end{proposition}

\begin{proof}[Proof of Proposition~\ref{theortetrcomp1}:]
Taking into account theorem~\ref{theortetr}, for  $B_{1}=B_{2}<0,$
$B_{3}=B_{4}<0$ we derive:

\begin{equation}\label{compl1}
\vec{B_{1}}+\vec{B_{2}}+\vec{B_{3}}+\vec{B_{4}}=\vec{0}
\end{equation}

or

\begin{equation}\label{compl2}
(-\vec{B_{1}})+(-\vec{B_{2}})+(-\vec{B_{3}})+(-\vec{B_{4}})=\vec{0}.
\end{equation}

From (\ref{compl1}) and (\ref{compl2}), we derive that the
complementary weighted Fermat-Torricelli point $A_{0}^{\prime}$
coincides with the weighted Fermat-Torricelli point $A_{0}.$ We
note that the vectors $\vec{B}_{i}$ may change direction from
$A_{i}$ to $A_{0},$ simultaneously, for $i=1,2,3,4.$

\end{proof}

\begin{proposition}\label{theortetrcomp2}
The location of the complementary weighted Fermat-Torricelli point
$A_{0}^{\prime}$ (solution of Problem~\ref{sym2complementary}) of
the regular tetrahedron $A_{1}A_{2}A_{3}A_{4}$ for
$B_{1}=B_{2}<0,$ $B_{3}=B_{4}>0$ or $B_{1}=B_{2}>0,$
$B_{3}=B_{4}<0$ and $\|B_{1}\|>\|B_{4}\|$ is given by:

\begin{eqnarray}\label{analsoltetrahedrcom}
&&y=\frac{\sqrt{t}}{2}+\nonumber\\
&&{}+\frac{1}{2} \sqrt{\frac{a^4 B_1^4}{4 s^{1/3}}-\frac{a^4 B_1^2
B_4^2}{2 s^{1/3}}+\frac{a^4 B_4^4}{4 s^{1/3}}-\frac{2 \left(-8
\sqrt{2} a^3 B_1^2-8 \sqrt{2} a^3 B_4^2\right)}{\sqrt{t} \left(64
B_1^2-64 B_4^2\right)}+\frac{s^{1/3}}{4 \left(B_1^4-2 B_1^2
B_4^2+B_4^4\right)}}\nonumber\\
\end{eqnarray}

\end{proposition}

\begin{proof}[Proof of Proposition~\ref{theortetrcomp2}:]

Considering (\ref{obj13}) for $B_{1}=B_{2}<0,$ $B_{3}=B_{4}>0$ or
$B_{1}=B_{2}>0,$ $B_{3}=B_{4}<0$ and $\|B_{1}\|>\|B_{4}\|$ and
differentiating (\ref{obj13}) with respect to $y\equiv
OA_{0}^{\prime},$ and by squaring both parts of the derived
equation, we obtain (\ref{fourth2}) which is a fourth order
equation with respect to $y.$ The second real solution of $y$
gives (\ref{analsoltetrahedrcom}). Taking into account the real
solution (\ref{analsoltetrahedrcom}) and the weighted floating
equilibrium condition
$\vec{B_{1}}+\vec{B_{2}}+\vec{B_{3}}+\vec{B_{4}}=\vec{0}$ we
obtain that the complementary weighted Fermat-Torricelli point
$A_{0}^{\prime}$ for $B_{1}=B_{2}<0,$ $B_{3}=B_{4}>0$ coincides
with the complementary weighted Fermat-Torricelli point
$A_{0}^{\prime\prime}$ for $B_{1}=B_{2}>0,$ $B_{3}=B_{4}<0$ (
Fig.~\ref{fig4} ). From (\ref{analsoltetrahedrcom}), we derive
that the complementary $A_{0}^{\prime}$ is located outside the
regular tetrahedron $A_{1}A_{2}A_{3}A_{4}$ (Fig.~\ref{fig4}).

\begin{figure}
\centering
\includegraphics[scale=0.2]{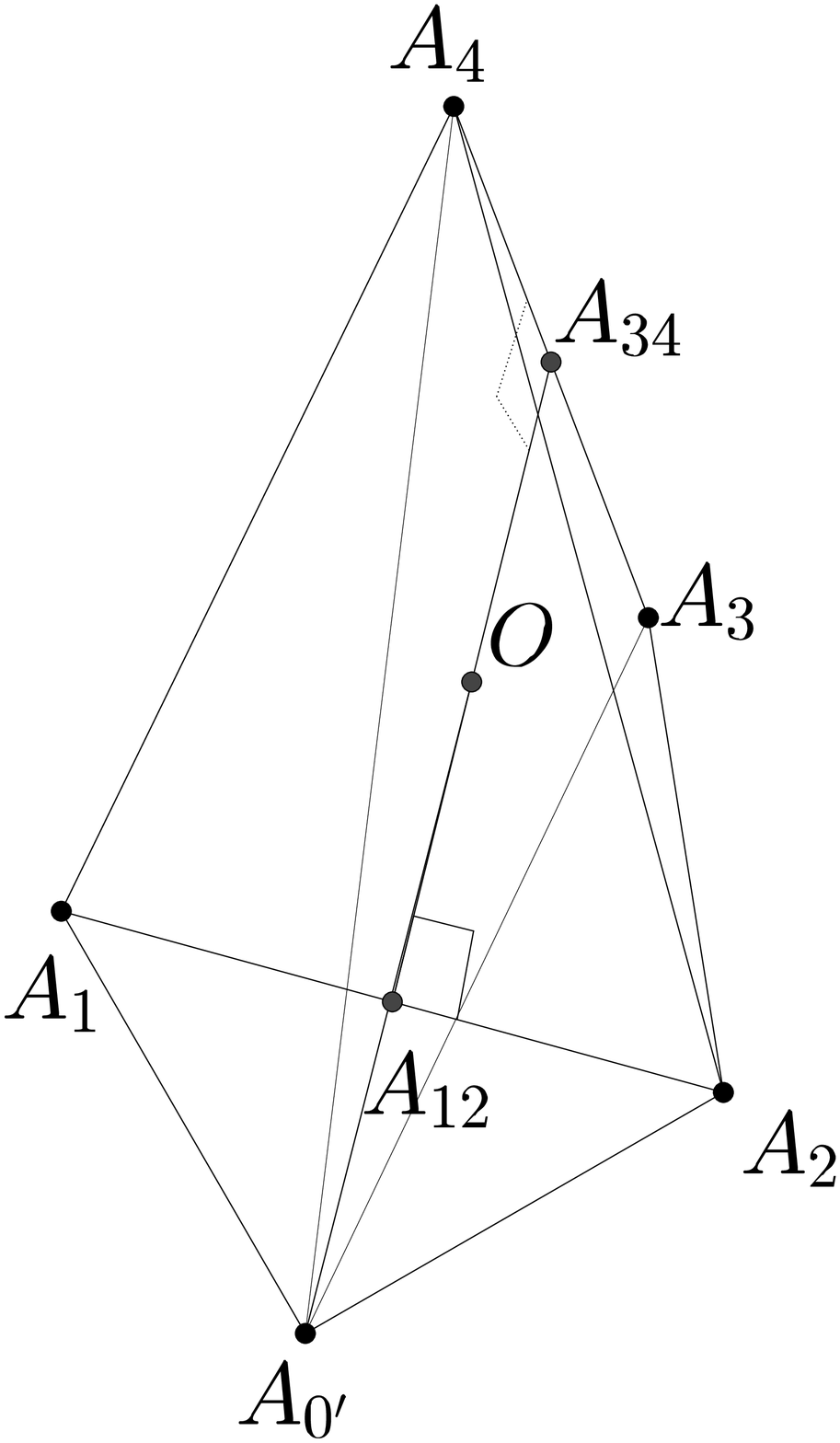}
\caption{The complementary weighted Fermat-Torricelli point
$A_{0}^{\prime}$ of a regular tetrahedron $A_{1}A_{2}A_{3}A_{4}$
for $B_{1}=B_{2}>0$ and $B_{3}=B_{4}<0$ or $B_{1}=B_{2}<0$ and
$B_{3}=B_{4}>0$  for $\|B_{1}\|>\|B_{4}\|$}\label{fig4}
\end{figure}

\end{proof}

\begin{example}\label{tetr1}
Given a regular tetrahedron $A_{1}A_{2}A_{3}A_{4}$ in
$\mathbb{R}^{3},$ $a=1, B_{1}=B_{2}=2.5,$ $B_{3}=B_{4}=1$ from
(\ref{analsolrtetrahedron}) and (\ref{analsoltetrahedrcom}) we get
$y=0.198358$ and $y=0.539791,$ respectively, with six digit
precision. The weighted Fermat-Torricelli point $A_{0}$ and the
complementary weighted Fermat-Torricelli point
$A_{0^{\prime}}\equiv A_{0}$ for $B_{1}=B_{2}=-2.5$ and
$B_{3}=B_{4}=-1$ corresponds to $y=0.198358.$ The complementary
weighted Fermat-Torricelli point $A_{0}^{\prime}$ for
$B_{1}=B_{2}=-2.5$ and $B_{3}=B_{4}=1$ or $B_{1}=B_{2}=1.5$ and
$B_{3}=B_{4}=-1$ lies outside the regular tetrahedron
$A_{1}A_{2}A_{3}A_{4}$ and corresponds to
$y=0.539791>\frac{A_{12}A_{34}}{2}=\frac{\sqrt{2}}{4}.$
\end{example}

We proceed by calculating the angles $\alpha_{i0j},$ for
$i,j=0,1,2,3,4.$

\begin{proposition}\label{anglestetragon}
The angles $\alpha_{i0j},$ for $i,j=0,1,2,3,4,$ are given by:

\begin{equation}\label{alpha102}
\alpha_{102}=\arccos{\left(1-\frac{a^2}{2\left(\left(\frac{a}{2}\right)^{2}+\left(\frac{\frac{a\sqrt{2}}{2}}{2}-y\right)^{2}\right)}\right)},
\end{equation}

\begin{equation}\label{alpha304}
\alpha_{304}=\arccos{\left(1-\frac{a^2}{2\left(\left(\frac{a}{2}\right)^{2}+\left(\frac{\frac{a\sqrt{2}}{2}}{2}+y\right)^{2}\right)}\right)},
\end{equation}

and

\begin{equation}\label{alpha401}
\alpha_{104}=\alpha_{203}=\alpha_{103}=\alpha_{204}=\arccos{\frac{\left(\frac{\frac{a\sqrt{2}}{2}}{2}-y\right)^{2}+\left(\frac{\frac{a\sqrt{2}}{2}}{2}-y\right)^{2}-\frac{a^{2}}{2}}{2
\sqrt{\left(\frac{a}{2}\right)^{2}+\left(\frac{\frac{a\sqrt{2}}{2}}{2}-y\right)^{2}}
\sqrt{\left(\frac{a}{2}\right)^{2}+\left(\frac{\frac{a\sqrt{2}}{2}}{2}+y\right)^{2}}}}.
\end{equation}

\end{proposition}

\begin{proof}[Proof of Proposition~\ref{anglestetragon}:]

Taking into account the cosine law in $\triangle A_{1}A_{0}A_{2},$
$\triangle A_{3}A_{0}A_{4},$ $\triangle A_{1}A_{0}A_{4},$
$\triangle A_{2}A_{0}A_{4},$ $\triangle A_{1}A_{0}A_{3},$
$\triangle A_{2}A_{0}A_{4},$ and (\ref{analsolrtetrahedron}), we
obtain (\ref{alpha102}), (\ref{alpha304}) and (\ref{alpha401}),
respectively.

\end{proof}

\begin{corollary}{\cite[Theorem~4.3, p.~102]{NSM:91}}\label{regular}
If $B_{1}=B_{2}=B_{3}=B_{4},$ then
\begin{equation}\label{regultetrsol}
\alpha_{i0j}=\arccos{\left(-\frac{1}{3}\right)},
\end{equation}
for $i,j=1,2,3,4$ and $i\ne j.$
\end{corollary}

\begin{proof}
By setting $y=0$ in (\ref{alpha102}), (\ref{alpha304}) and
(\ref{alpha401}), we obtain (\ref{regultetrsol}).
\end{proof}


\section{The weighted Fermat-Torricelli problem for tetrahedra in the three dimensional Euclidean Space: The case  $B_{1}=B_{2}$ and $B_{3}=B_{4}.$ }

We consider the following lemma which gives the invariance
property (geometric plasticity) of the weighted Fermat-Torricelli
point for a given tetrahedron
$A_{1}^{\prime}A_{2}^{\prime}A_{3}^{\prime}A_{4}^{\prime}$ in
$\mathbb{R}^{3}$ (\cite[Appendix~AII,pp.~851-853]{ZachosZu:11})

\begin{lemma}{\cite[Appendix~AII,pp.~851-853]{ZachosZu:11}}\label{tetragonnntetrah}
Let $A_1A_2A_{3}A_4$ be a regular tetrahedron in $\mathbb{R}^{3}$
and each vertex $A_{i}$ has a non-negative weight $B_{i}$ for
$i=1,2,3,4.$ Assume that the floating case of the weighted
Fermat-Torricelli point $A_{0}$ occurs:
\begin{equation}\label{floatingcasetetr1}
\|{\sum_{j=1, i\ne j}^{4}B_{j}\vec u(A_i,A_j)}\|>B_i.
\end{equation}
If $A_0$ is connected with every vertex $A_i$ for $i=1,2,3,4$ and
a point $A_{i}^{\prime}$ is selected with corresponding
non-negative weight $B_{i}$ on the ray that is defined by the line
segment $A_0A_i$ and the tetrahedron
$A_{1}^{\prime}A_{2}^{\prime}A_{3}^{\prime}A_{4}^{\prime}$ is
constructed such that:

\begin{equation}\label{floatingcasequad2}
\|{\sum_{j=1, i\ne j}^{4}B_{j}\vec
u(A_{i}^{\prime},A_{j}^{\prime})}\|>B_i,
\end{equation}
then the weighted Fermat-Torricelli point $A_{0}^{\prime}$ of
$A_{1}^{\prime}A_{2}^{\prime}A_{3}^{\prime}A_{4}^{\prime}$ is
identical with $A_{0}.$
\end{lemma}


We consider a tetrahedron $A_{1}A_{2}A_{3}A_{4}^{\prime}$ which
has as a base the equilateral triangle  $\triangle
A_{1}A_{2}A_{3}$ with side $a$ and the vertex $A_{4}^{\prime}$ is
located on the ray $A_{0}A_{4},$  with corresponding non-negative
weights $B_{1}=B_{2}$ at the vertices $A_{1}, A_{2}$ and
$B_{3}=B_{4}$ at the vertices $A_{3}, A_{4}^{\prime}.$

Assume that we choose  $B_{1}$ and $B_{4}$  non negative weights
which satisfy the inequalities (\ref{floatingcasetetr1}),
(\ref{floatingcasequad2}) and $B_{1}>B_{4},$ which correspond to
the weighted floating case of $A_{1}A_{2}A_{3}A_{4}$ and
$A_{1}A_{2}A_{3}A_{4}^{\prime}.$

We denote by $a_{i4^{\prime}}$ the length of the linear segment
$A_{i}A_{{4}^{\prime}},$ the angle $\angle
A_{i}A_{k}A_{{4}^{\prime}}$ for $i,j,0,1,2,3,4, i\neq j,$ by
$h_{0,12}$ the height of $\triangle A_{0}A_{1}A_{2}$ from $A_{0}$
to $A_{1}A_{2},$ by $\alpha$ the dihedral angle between the planes
$A_{0}A_{1}A_{2}$ and $A_{3}A_{1}A_{2}$ and by
$\alpha_{g_{4^{\prime}}}$ the dihedral angle between the planes
$A_{3}A_{1}A_{2}$ and $A_{4^{\prime}}A_{1}A_{2}$ and by $A_{0}$
the corresponding weighted Fermat-Torricelli point of the regular
tetrahedron $A_{1}A_{2}A_{3}A_{4}.$

\begin{theorem}\label{theorquadnntetrah}
The location of the weighted Fermat-Torricelli point
$A_{0^{\prime}}$ of a tetrahedron $A_{1}A_{2}A_{3}A_{4}^{\prime}$
which has as a base the equilateral triangle  $\triangle
A_{1}A_{2}A_{3}$ with side $a$ and the vertex $A_{4}^{\prime}$ is
located on the ray $A_{0}A_{4}$ for $B_{1}=B_{2}$ and
$B_{3}=B_{4},$ under the conditions (\ref{floatingcasetetr1}),
(\ref{floatingcasequad2}) and $B_{1}>B_{4},$ is given by:

\begin{equation}\label{a04prime}
a_{04^{\prime}}=\sqrt{a_{20}^{2}+a_{24^{\prime}}^2-2
a_{24^{\prime}}\left(\sqrt{a_{02}^2-h_{0,12}^2}\cos\alpha_{124^{\prime}}+h_{0,12}\sin\alpha_{124^{\prime}}\cos(\alpha_{g_{4^{\prime}}}-\alpha)\right)}
\end{equation}

where

\begin{equation}\label{alpha}
\alpha=\arccos{\frac{\frac{a_{02}^2+a_{23}^2-a_{03}^2}{2
a_{23}}-\sqrt{a_{2}^2-h_{0,12}^2}\cos\alpha_{123}}{h_{0,12}\sin\alpha_{123}}}
\end{equation}

and

\begin{equation}\label{height012}
h_{0,12}=\sqrt{\frac{4a_{01}^{2}a_{02}^2-(a_{01}^2+a_{02}^2-a_{12}^2)^2}{4
a_{12}^2}}.
\end{equation}

\end{theorem}

\begin{proof}[Proof of Theorem~\ref{theorquadnntetrah}:]

From lemma~\ref{tetragonnntetrah}, we get $A_{0^{\prime}}\equiv
A_{0}.$ Therefore, we get the relations (\ref{a04prime}) and
(\ref{alpha}) from a generalization of the cosine law in
$\mathbb{R}^{3}$ which has been introduced for tetrahedra in
\cite[Solution of Problem~1, Formulas (2.14) and
(2.20),p.~116]{Zachos/Zou:09}.

\end{proof}

\begin{remark}
We may consider a tetrahedron
$A_{1}A_{2}A_{3^{\prime}}A_{4^{\prime}}$ by placing
$A_{3^{\prime}}$ on the ray defined by $A_{0}A_{3}$ and
$A_{0^{\prime \prime}}$ is the corresponding weighted
Fermat-Torricelli point. Taking into account
lemma~\ref{tetragonnntetrah}, we get $A_{0^{\prime\prime}}\equiv
A_{0^{\prime}}\equiv A_{0}.$
\end{remark}


The author is sincerely grateful to Professor Dr. Vassilios G.
Papageorgiou for his very valuable comments, many fruitful
discussions and for bringing my attention to this particular
problem.

\end{document}